\documentclass[preprint]{elsarticle}

\usepackage{amssymb}
\usepackage{amsmath}
\usepackage{amsthm}
\usepackage{enumerate}
\usepackage{graphicx}
\usepackage{amsfonts}
\usepackage{minipage}
\usepackage{float}
\newtheorem{tb}{Table}
\newtheorem{example}{Example}

\journal{...}

\newtheorem{theorem}{Theorem}[section]

\newtheorem{lemma}{Lemma}[section]

\newtheorem{remark}{Remark}[section]

\theoremstyle{definition}

\begin{document}

	\begin{frontmatter}

		\title{Better approximation of functions by genuine Bernstein-Durrmeyer type
			operators}
		
		\author[1]{Ana-Maria Acu}
		\author[2]{Purshottam Agrawal}
		\address[1]{Lucian Blaga University of Sibiu, Department of Mathematics and Informatics, Str. Dr. I. Ratiu, No.5-7, RO-550012  Sibiu, Romania, e-mail: anamaria.acu@ulbsibiu.ro}
		\address[2]{Department of Mathematics,
			Indian Institute of Technology Roorkee,
			Roorkee-247667,  India, e-mail: pnappfma@gmail.com}

\begin{abstract}
The main object of this paper is to construct a new genuine Bernstein-Durrmeyer type operators which have better features than the classical one. Some direct estimates for the modified genuine Bernstein-Durrmeyer operator by means of the first and second modulus of continuity are given. An asymptotic formula for the new operator is proved. Finally, some numerical examples with illustrative graphics have been added to validate the theoretical results and also compare the rate of convergence.  
\end{abstract}


\end{frontmatter}

\section{Introduction}

Bernstein operators are one of the most important sequences of positive linear operators. These operators were introduced by Bernstein
\cite{Bernstein} and were intensively studied.  For more details on this topic we can refer the readers to excellent monographs \cite{G1} and \cite{G2}.  
The
Bernstein operators  are given by
\begin{equation}\label{X} B_{n}:C[0,1]\to C[0,1], \quad B_{n}(f;x)=\displaystyle\sum_{k=0}^n f\left(\frac{k}{n}\right)p_{n,k}(x),  \end{equation}
where
$$ p_{n,k}(x)={n \choose k}x^k(1-x)^{n-k},\,\,  x\in[0,1].$$
It is well known that the fundamental polynomials verify
\begin{equation}\label{et1}
p_{n,k}(x) = (1-x) \ p_{n-1,k}(x) + x \ p_{n-1,k-1}(x), \ \ 0<k< n.
\end{equation}
In a recent paper,   Khosravian-Arab et al. \cite{1} have introduced a sequence of modified Bernstein operators  as follows:
\begin{eqnarray}\label{ne1}
B_n^{M,1}(f,x) = \sum\limits_{k=0}^{n} p_{n,k}^{M,1}(x) \ f\left(\frac{k}{n} \right) , \ \ x \in [0,1],
\end{eqnarray}
\begin{align}\label{et2}
p_{n,k}^{M,1}(x) &= \alpha(x,n) \ p_{n-1,k}(x) + \alpha(1-x,n) \ p_{n-1,k-1}(x), 1\leq k\leq n-1,\\
p_{n,0}^{M,1}(x) &= \alpha(x,n)(1-x)^{n-1},\quad p_{n,n}^{M,1}(x) = \alpha(1-x,n)x^{n-1},\notag
\end{align}
and
\begin{eqnarray*}
	\alpha(x,n) = \alpha_1(n) \ x + \alpha_0(n), \,  n=0,1, \ldots,
\end{eqnarray*}
where $\alpha_0(n)$ and $\alpha_1(n)$ are  unknown sequences.  For $\alpha_1(n)=-1,$ $\alpha_0(n)=1$, obviously, (\ref{et2}) reduces to (\ref{et1}).

A Kantorovich variant of  the modified
Bernstein operators (\ref{ne1}) was introduced and studied in \cite{ana3}.

\section{The modified genuine Bernstein-Durrmeyer operators }
The genuine Bernstein-Durrmeyer operators were introduced by Chen \cite{D2} and Goodman and Sharma \cite{dif_GBD}
 and were studied  by a numbers of
authors (see \cite{ana}, \cite{dif_A2}, \cite{79}, \cite{80},  \cite{ana1}, \cite{143}). These operators are defined as follows:
\begin{align*} &{ U}_{n}(f;x)=(1-x)^n f(0)+ x^n f(1)
+(n-1)
\displaystyle\sum_{k=1}^{n-1}\left(\int_{0}^1 f(t)p_{n-2,k-1}(t)dt\right)p_{n,k}(x),\,\,  f\in C[0,1]. \end{align*}
The genuine Bernstein-Durrmeyer operators are limits of the Bernstein-Durrmeyer operators with Jacobi weights (see \cite{D11}, \cite{D12}, \cite{D100}), namely
$$  U_nf=\displaystyle\lim_{\alpha\to-1, \beta\to -1} M_n^{<\alpha,\beta>}f,\textrm{where} $$
\begin{align*}
&M_n^{<\alpha,\beta>}:C[0,1]\to\Pi_n,\,\, M_n^{<\alpha,\beta>}(f;x)=\displaystyle\sum_{k=0}^n p_{n,k}(x)\frac{\int_0^1w^{(\alpha,\beta)}(t)p_{n,k}(t)f(t)dt}{\int_0^1w^{(\alpha,\beta)}(t)p_{n,k}(t)dt},\\
& w^{(\alpha,\beta)}(t)=x^{\beta}(1-x)^{\alpha},\,x\in(0,1),\,\,\alpha,\beta>-1.
\end{align*}

In this section we introduce a new variant of the genuine Bernstein-Durrmeyer operators as follows:
\begin{align}\label{e1}
&{ U}_n^{1}(f;x)=\alpha(x,n)(1-x)^{n-1}f(0)+\alpha(1-x,n)x^{n-1}f(1)\\
&+(n-1)\displaystyle\sum_{k=1}^{n-1}\left\{\alpha(x,n)p_{n-1,k}(x)+\alpha(1-x,n)p_{n-1,k-1}(x)\right\}\int_0^1p_{n-2,k-1}(t)f(t)dt.\nonumber
\end{align}
Throughout this section we assume $U_n^{1}(e_0)=1$, namely
\begin{equation}\label{e2} 2\alpha_0(n)+\alpha_1(n)=1.\end{equation}
In the following we will consider these two cases:
\begin{equation}\label{e3}
\alpha_0(n)\geq 0, \alpha_0(n)+\alpha_1(n)\geq 0,
\end{equation}
\begin{equation}\label{e4} \alpha_0(n)<0 \textrm{ or } \alpha_1(n)+\alpha_0(n)<0.
\end{equation}
\begin{remark} If the unknown sequences $\alpha_i(n),i=1,2$ verify conditions (\ref{e2}) and (\ref{e3}), it follows that
	$$ 0\leq \alpha_0(n)\leq 1\textrm{ and } -1\leq \alpha_1(n)\leq 1. $$
Thus the sequences $\alpha_i(n),i=1,2$ are bounded and the operator (\ref{e1}) is positive. If the  sequences $\alpha_i(n),i=1,2$ verify conditions (\ref{e2}) and (\ref{e4}), we obtain
	$$ \big(\alpha_0(n)<0,\,\,\alpha_1(n)+\alpha_0(n)>1\big)\textrm{ or }\big(\alpha_1(n)+\alpha_0(n)<0,\,\, \alpha_0(n)>1\big), $$
	hence the operator (\ref{e1}) is not positive.
\end{remark}
Denote $m_{n,k}^{1}(x):={ U}_n^{1}(e_k;x)$, $\mu_{n,k}^{1}(x):={ U}_n^{1}\left((t-x)^k;x\right)$, where $e_k(t)=t^k$, $k=0,1,\dots.$
\begin{lemma}\label{l1}
	The modified genuine Bernstein-Durrmeyer operators (\ref{e1}) verify:
	\begin{itemize}
		\item[i)] $m_{n,0}^{1}(x)=1$;
		\item[ii)] $m_{n,1}^{1}(x)=\displaystyle\frac{1}{n}\left\{xn+(1-2x)(1-\alpha_0(n))\right\}$;
		\item[iii)] $m_{n,2}^{1}(x)=\displaystyle\frac{1}{n(n+1)}\left\{x^2n^2+(4\alpha_0(n)x-2\alpha_0(n)-5x+4)xn\right.
		\left. +2(1-x)(1-2x)(1-\alpha_0(n))\right\}$.
		\end{itemize}
\end{lemma}
\begin{lemma}\label{l2}
	The following statements hold:
	\begin{itemize}
		\item[i)] $\mu_{n,1}^{1}(x)=\displaystyle\frac{1}{n}(1-2x)(1-\alpha_0(n))$;
		\item[ii)] $\mu_{n,2}^{1}(x)=\displaystyle\frac{2}{n(n+1)}\left\{x(1-x)n+(2x-1)^2(1-\alpha_0(n))\right\}$;
		\item[iii)]$\mu_{n,4}^{1}(x)=\displaystyle\frac{12}{n(n+1)(n+2)(n+3)}\left\{x^2(1-x)^2n^2-x(1-x)n\right.\\
		\left.\cdot\left[4\alpha_0(n)(1-2x)^2+23x(1-x)-6\right]+2(1-2x)^4\left[1-\alpha_0(n)\right]
		\right\}.$
	\end{itemize}
\end{lemma}

In the following we will give a direct estimate for the modified genuine Bernstein-Durrmeyer operator ${ U}_n^{1}$ by means of the first modulus of continuity $\omega(f,\delta)$.
\begin{theorem}
	Let $f$ be a bounded function for $x\in[0,1]$. If $\alpha_1(n)$ is a bounded sequence, then
	$$ \|{ U}_n^{1}f-f\|\leq 2 \left(3|\alpha_1(n)|+1\right)\omega\left(f;\frac{1}{\sqrt{n}}\right) , $$
	where $\|\cdot\|$ is the uniform norm on the interval $[0,1]$. 
\end{theorem}
\begin{proof}
	It is known that the modulus of continuity $\omega(f;\delta)$ verifies
	\begin{align}\label{e7} |f(t)-f(x)|\leq \omega(f;\delta)\left(\displaystyle\frac{(t-x)^2}{\delta^2}+1\right). \end{align}
	It follows from (\ref{e2}) that
	\begin{align}\label{ec23}
	|\alpha(x,n)|&=|\alpha_1(n)x+\alpha_0(n)|\leq |\alpha_1(n)|+|\alpha_0(n)|\nonumber\\
	&=|\alpha_1(n)|+\left|\frac{1-\alpha_1(n)}{2}\right|\leq \frac{3}{2}|\alpha_1(n)|+\frac{1}{2}.
	\end{align}
	In a similar way, we get
	$$ |\alpha(1-x,n)|\leq\dfrac{3}{2} |\alpha_1(n)|+\dfrac{1}{2}.$$
	 Therefore,
	using Lemma \ref{l1}, condition (\ref{e2}) and 	chosing $\delta=\displaystyle\frac{1}{\sqrt{n}}$, we get
	\begin{align*} &\left|{U}_n^{1}(f;x)-f(x)\right|\\
	&\leq |{\alpha(x,n)}|(1-x)^{n-1}|f(0)-f(x)| +|\alpha(1-x,n)|x^{n-1}|f(1)-f(x)|\\
	&+(n-1)\displaystyle\sum_{k=1}^{n-1}|\alpha(x,n)|p_{n-1,k}(x)\int_0^1p_{n-2,k-1}(t)|f(t)-f(x)|dt\\
	&+(n-1)\displaystyle\sum_{k=1}^{n-1}|\alpha(1-x,n)|p_{n-1,k-1}(x)\int_0^1p_{n-2,k-1}(t)|f(t)-f(x)|dt\\
	&\leq\left(\displaystyle\frac{3}{2}|\alpha_1(n)|+\frac{1}{2}\right)\omega\left(f;\frac{1}{\sqrt{n}}\right)\left\{(1-x)^{n-1}(nx^2+1)+x^{n-1}(n(1-x)^2+1)\right.\\
	&+(n-1)\displaystyle\sum_{k=1}^{n-1}p_{n-1,k}(x)\int_0^1p_{n-2,k-1}(t)\left(n(t-x)^2+1\right)dt\\
	&+\left.(n-1)\displaystyle\sum_{k=1}^{n-1}p_{n-1,k-1}(x)\int_0^1p_{n-2,k-1}(t)\left(n(t-x)^2+1\right)dt\right\}\\
	&=\displaystyle\frac{1}{n+1}\left(\displaystyle 3|\alpha_1(n)|+1\right)\omega\left(f;\frac{1}{\sqrt{n}}\right)\left[(-2x^2+2x+1)n+4x^2-4x+2\right]\\
	&\leq 2\left(\displaystyle 3|\alpha_1(n)|+1\right)\omega\left(f;\frac{1}{\sqrt{n}}\right).
	\end{align*}
\end{proof}
\begin{remark}\label{r1} If $f\in C[0,1]$, then $\displaystyle\lim_{n\to\infty}\omega\left(f,\frac{1}{\sqrt{n}}\right)=0$. Therefore,  $({ U}_n^{1})_n$ converges uniformly on $[0,1]$ for   $\alpha_1(n)$  a bounded sequence  and $f\in C[0,1]$.
\end{remark}

\begin{theorem}\label{t2}
	Let $\alpha_i(n),i=0,1$ be  convergent sequences that satisfy the conditions (\ref{e3}) and $l_0=\displaystyle\lim_{n\to\infty}\alpha_0(n)$.  If $f^{\prime\prime}\in C[0,1]$, then
	$$
	\displaystyle\lim_{n\to\infty} n\left({ U}_n^{1}(f;x)-f(x)\right)=(1-2x)(1-l_0)f^\prime(x)+x(1-x) f^{\prime\prime}(x),$$
	uniformly on $[0,1]$.
\end{theorem}
\begin{proof}
	Applying the modified genuine Bernstein-Durrmeyer operators ${ U}_n^{1}$ to the Taylor's formula, we obtain
	\begin{align*}{U}_n^{1}(f;x)-f(x)&={ U}_n^{1}(t-x;x)f^{\prime}(x)+\displaystyle\frac{1}{2}{ U}_n^{1}\left((t-x)^2;x\right)f^{\prime\prime}(x)\\
	&+{ U}_n^{1}\left(\xi(t,x)(t-x)^2;x\right), \end{align*}	
	where $\xi\in C[0,1]$ and $\displaystyle\lim_{t\to x}\xi(t,x)=0$.
	
	Using the Cauchy-Schwarz inequality, we get
	$$ n{ U}_n^{1}\left(\xi(t,x)(t-x)^2;x\right)\leq\sqrt{{ U}_n^{1}\left(\xi^2(t,x);x\right)} \sqrt{{n^2 U}_n^{1}\left((t-x)^4;x\right)}. $$
	From Lemma \ref{l2} we have $\displaystyle\lim_{n\to\infty}n^2{ U}_n^{1}\left((t-x)^4;x\right)=12x^2(1-x)^2$.	Since\linebreak $\xi^2(x,x)=0$ and $\xi^2(\cdot,x)\in C[0,1]$, by Remark \ref{r1}, we obtain
	$$ \displaystyle\lim_{n\to\infty}{ U}_n^{1}\left(\xi^2(t,x);x\right)=0 $$
	uniformly with respect to $x\in[0,1]$. Therefore, 
	$$ \displaystyle\lim_{n\to\infty}n{ U}_n^{1}\left(\xi(t,x)(t-x)^2;x\right)=0. $$
	Applying Lemma \ref{l2}, the theorem is proved.
\end{proof}

Now, we extend the results from  Theorem \ref{t2} when the sequences $\alpha_i(n)$, $i=0,1$, satisfy the conditions  (\ref{e4}), namely the operator ${ U}_n^{1}$ is nonpositive.

\begin{theorem}\label{t2.5} Let $\alpha_i(n), i=0,1$ be bounded convergent sequences which satisfy  \eqref{e4} and $l_0=\displaystyle\lim_{n\to\infty} \alpha_0(n)$. If $f\in C[0,1]$ and $f^{\prime\prime}$ exists at a certain point $x\in[0,1]$, then we have
	\begin{equation}\label{e5}
	\displaystyle\lim_{n\to\infty} n\left({ U}_n^{1}(f;x)-f(x)\right)=(1-2x)(1-l_0)f^\prime(x)+x(1-x) f^{\prime\prime}(x),\end{equation}
	Moreover the relation \eqref{e5} holds uniformly on $[0,1]$ if $f^{\prime\prime}\in C[0,1].$
\end{theorem}
\begin{proof}
	Applying the modified genuine Bernstein-Durrmeyer operators ${ U}_n^{1}$ to the Taylor's formula, we get
\begin{align*} {U}_n^{1}(f;x)-f(x)&={ U}_n^{1}(t-x;x)f^{\prime}(x)+\displaystyle\frac{1}{2}{ U}_n^{1}\left((t-x)^2;x\right)f^{\prime\prime}(x)\\
&+{ U}_n^{1}\left(\xi(t,x)(t-x)^2;x\right),  \end{align*}	
	where $\xi\in C[0,1]$ and $\displaystyle\lim_{t\to x}\xi(t,x)=0$.
	It is sufficient to show that
	\begin{align}
	\displaystyle\lim_{n\to\infty} n{ U}_{n}^{1}(\xi(t,x)(t-x)^2;x)&=0.\label{eG2}
	\end{align}
	Since the operators $ { U}_n^{1}$ are not positive linear operators  we can not use Cauchy-Schwarz inequality and we will introduce new techniques in order to prove the theorem.
	
	Let $\varepsilon>0$ be given. There exist a $\delta>0$ such that if $|t-x|<\delta$ then $|\xi(t,x)|<\varepsilon$. We denote
	\begin{align*}
	&
	{ K}_1=\left\{i:\left|\frac{i}{n}-x\right|<\delta, i=0,1,2,\cdots ,n\right\},\\
	& { K}_2=\left\{i:\left|\frac{i}{n}-x\right|\ge \delta, i=0,1,2,\cdots ,n\right\}.\end{align*}
	The boundedness of the sequences $\alpha_i(n)$, $i=0,1$ implies that these is a constant $C>0$ such that $|\alpha(x,n)|<C$. 
	
	Let $k\in { K}_1$. Hence $|\xi(t,x)|<\varepsilon$. Therefore we get	
	\begin{align}
	&\left|{ U}_n^{1}(\xi(t,x)(t-x)^2;x)\right|\leq C(1-x)^{n-1}|\xi(0,x)|x^2+Cx^{n-1}|\xi(1,x)|(1-x)^2\nonumber\\
	&+ C(n-1)\displaystyle\sum_{k=1}^{n-1}\left\{p_{n-1,k}(x)+p_{n-1,k-1}(x)
	\right\}\int_0^1p_{n-2,k-1}(t)|\xi(t,x)|(t-x)^2dt\nonumber\\
		&\leq C\varepsilon \left\{(n-1)\displaystyle\sum_{k=1}^{n-1}\left[p_{n-1,k}(x)+p_{n-1,k-1}(x)\right]\int_0^1p_{n-2,k-1}(t)(t-x)^2 dt\right.\nonumber\\
	&\left.+(1-x)^{n-1}x^2+x^{n-1}(1-x)^2\right\}=\displaystyle\frac{2\varepsilon C}{n(n+1)}\left\{2nx(1-x)+(1-2x)^2\right\}.
	\label{e6}
	\end{align}
	
	Let $k\in { K}_2$. We denote $M=\displaystyle\sup_{0\leq t\leq 1}|\xi(t,x)|(t-x)^2$. Then $$|\xi(t,x)|(t-x)^2\leq\displaystyle \frac{M}{\delta^4}\left(\frac{k}{n}-x\right)^4$$
	Moreover, the below upper bound is obtained 
	\begin{align}\label{e71}
	&\left|{ U}_n^{1}(\xi(t,x)(t-x)^2;x)\right|\\
	&\leq \displaystyle\frac{MC}{\delta^4}\left\{\displaystyle\sum_{k=1}^{n-1}\left[p_{n-1,k}(x)+p_{n-1,k-1}(x)
	\right]\left(\displaystyle\frac{k}{n}-x\right)^4+(1-x)^{n-1}x^4+x^{n-1}(1-x)^4\right\}\nonumber\\
	&=\displaystyle\frac{CM}{n^4\delta^4}\left\{6x^2(1\!-x)^2n^2+4x(1\!-x)(13x^2\!-13x\!+3)n\!+(1\!-2x)^2(12x^2-12x+1)\right\}.\nonumber
	\end{align}
	Using (\ref{e6}) and (\ref{e71}), it follows
	\begin{align*}
	&\left|{ U}_n^{1}(\xi(t,x)(t-x)^2;x)\right|\leq\displaystyle\frac{2\varepsilon C}{n(n+1)}\left\{2nx(1-x)+(1-2x)^2\right\}\nonumber\\
	&+\displaystyle\frac{CM}{n^4\delta^4}\left\{6x^2(1\!-x)^2n^2\!+4x(1\!-x)(13x^2\!-13x+3)n\!+(1\!-2x)^2(12x^2\!-12x+1)\right\}.
	\end{align*}
	Therefore, the last inequality leads to (\ref{eG2}) and the theorem is proved.
\end{proof}

\section{Better rate of convergence}

In the following we will improve the previous results considering a new genuine Bernstein-Durrmeyer operators that have order of approximation $\displaystyle{ O}\left(n^{-2}\right)$ defined as

\begin{align}\label{U2}
{ U}_n^{2}(f;x)&=(n-1)\displaystyle\sum_{k=1}^{n-1} p_{n,k}^{2}(x)\int_0^1p_{n-2,k-1}(t)f(t)dt\nonumber\\
&+\beta(x,n)(1-x)^{n-2}f(0)+\beta(1-x,n)x^{n-2}f(1).
\end{align}
where
\begin{equation}
\label{e38p}
p_{n,k}^{2}(x)=\beta(x,n)p_{n-2,k}(x)+\gamma(x,n)p_{n-2,k-1}(x)+\beta(1-x,n)p_{n-2,k-2}(x) \end{equation}
and
$$ \beta(x,n)=\beta_2(n)x^2+\beta_1(n)x+\beta_0(n),\,\,\gamma(x,n)=\gamma_0(n)x(1-x),  $$
where $\beta_i(n),i=0,1,2$ and $\gamma_0(n)$ are  unknown sequences. For $\beta_2(n)=\beta_0(n)=1,$ $\beta_1(n)=-2$, $\gamma_0(n)=2$ we obtained the classical genuine-Bernstein-Durrmeyer operators.

In the following we suppose ${ U}_n^{2}(e_0;x)=1$, namely
$$ \left(2\beta_2(n)-\gamma_0(n)\right)x^2-\left(2\beta_2(n)-\gamma_0(n)\right)x+2\beta_0(n)+\beta_1(n)+\beta_2(n)=1. $$ This yields
$$
2\beta_2(n)-\gamma_0(n)=0,\,\,
2\beta_0(n)+\beta_1(n)+\beta_2(n)=1.
$$
The above relations lead to
\begin{align*}
{ U}_n^{2}(e_1;x)&=x+\displaystyle\frac{2}{n}\left[2\left(\beta_0(n)-1\right)x-\beta_0(n)+1\right];\\
{ U}_n^{2}(e_2;x)&=x^2+\displaystyle\frac{2}{n(n+1)}\left[4\beta_0(n)nx^2-2\beta_0(n)nx-8\beta_0(n)x^2+\beta_2(n)x^2\right.\\
&-5nx^2+10\beta_0(n)x
-\left. \beta_2(n)x+3nx+7x^2-3\beta_0(n)-9x+3 \right].
\end{align*}
In order to have $\displaystyle\lim_{n\to\infty}{ U}_n^{2}(e_i;x)=x^i$, $i=0,1,2$ we suppose the sequences $\beta_0(n)$ and $\beta_2(n)$ to verify the conditions
$$ \displaystyle\lim_{n\to\infty}\frac{\beta_0(n)}{n}=0\textrm{ and } \displaystyle\lim_{n\to\infty}\frac{\beta_2(n)}{n^2}=0.  $$
We consider  the case $\beta_0(n)=1$ and $\beta_2(n)=n$. Thus  $\beta_1(n)=-n-1$ and $\gamma_0(n)=2n$.
For this particular case the modified genuine Bernstein-Durrmeyer operator becomes
\begin{align}\label{TK2}
\tilde{ U}_n^{2}(f;x)&=(n-1)\displaystyle\sum_{k=1}^{n-1} \tilde{p}_{n,k}^{2}(x)\int_0^1p_{n-2,k-1}(t)f(t)dt\\
&+\left[nx^2-(n+1)x+1\right](1-x)^{n-2}f(0)+\left[nx^2-(n-1)x\right]x^{n-2}f(1),\nonumber
\end{align}
where
\begin{align*}
\tilde{p}_{n,k}^{2}(x)&=\left[nx^2-(n+1)x+1\right]p_{n-2,k}(x)+2nx(1-x)p_{n-2,k-1}(x)\\
&+\left[nx^2-(n-1)x\right]p_{n-2,k-2}(x). \end{align*}

Denote $\tilde{m}_{n,k}^{2}(x):=\tilde{ U}_n^{2}(e_k;x)$, $\tilde{\mu}_{n,k}^{2}(x):=\tilde{ U}_n^{2}\left((t-x)^k;x\right)$, $k=0,1,\dots.$ the moments and central moments, respectively for  the modified genuine Bernstein-Durrmeyer operator $\tilde{ U}_n^{2}$.

\begin{lemma}
	The following statements hold:
	\begin{itemize}
		\item[i)] $\tilde{m}_{n,0}^{2}(x)=1$;
		\item[ii)] $\tilde{m}_{n,1}^{2}(x)=x$;
		\item[iii)] $\tilde{m}_{n,2}^{2}(x)=x^2+\displaystyle\frac{2x(1-x)}{n(n+1)}$.
	\end{itemize}
\end{lemma}
\begin{lemma}
	The following statements hold:
	\begin{itemize}
		\item [i)] $\tilde{\mu}_{n,2}^{2}(x)=\displaystyle\frac{2x(1-x)}{n(n+1)}$,
		\item[ii)] $\tilde{\mu}_{n,3}^{2}(x)=-\displaystyle\frac{6x(1-x)(1-2x)(n-2)}{n(n+1)(n+2)}$,
		\item[iii)] $\tilde{\mu}_{n,4}^{2}(x)=-\displaystyle\frac{12x^2(1-x)^2n}{(n+1)(n+2)(n+3)}+{ O}\left(n^{-3}\right)$,
		\item[iv)] $\tilde{\mu}_{n,5}^{2}(x)=\displaystyle\frac{240x^2(1-x)^2(2x-1)n}{(n+1)(n+2)(n+3)(n+4)}+{ O}\left(n^{-4}\right)$,
		\item[v)] $\tilde{\mu}_{n,6}^{2}(x)=-\displaystyle\frac{240x^3(1-x)^3n^2}{(n+1)(n+2)(n+3)(n+4)(n+5)}+{ O}\left(n^{-4}\right)$.
	\end{itemize}
\end{lemma}
\begin{theorem}
	If $f\in C^{6}[0,1]$ and $x\in [0,1]$, then 
	$$ \tilde{ U}_n^{2}(f;x)-f(x)={ O}\left(n^{-2}\right). $$
\end{theorem}
\begin{proof} Applying the modified genuine Bernstein-Durrmeyer operators $\tilde{ U}_n^{2}$ to the Taylor's formula, we obtain
	$$ \tilde{ U}_n^{2}(f;x)-f(x)=\displaystyle\sum_{k=1}^6 f^{(k)}(x)\tilde{ U}_n^{2}\left( (t-x)^k;x\right)+\tilde{ U}_n^{2}\left(\xi(t,x)(t-x)^6;x\right),  $$
	where $\displaystyle\lim_{t\to x}\xi(t,x)=0$.	
	We have
	\begin{align}
&	\tilde{ U}_n^{2}(f;x)=(n-1)\left(nx^2-(n+1)x+1\right)\sum_{k=1}^{n-2}p_{n-2,k}(x)\int_0^1p_{n-2,k-1}(t)f(t)dt\nonumber\\
	&+2n(n-1)x(1-x)\sum_{k=1}^{n-1}p_{n-2,k-1}(x)\int_0^1p_{n-2,k-1}(t)f(t)dt\nonumber\\
	&+(n-1)\left(nx^2-(n-1)x\right)\sum_{k=2}^{n-1}p_{n-2,k-2}(x)\int_0^1p_{n-2,k-1}(t)f(t)dt\nonumber\\
	&+\left(nx^2-(n+1)x+1\right)(1-x)^{n-2}f(0)+\left(nx^2-(n-1)x\right)x^{n-2}f(1).
	\label{xx}
	\end{align}	
	Let $\varepsilon>0$ be given. There exist  $\delta>0$ such that if $|t-x|<\delta$, then $|\xi(t,x)|<\varepsilon$. We denote
	\begin{align*}
	&{ K}_1=\left\{i:\left|\frac{i}{n}-x\right|<\delta, i=0,1,2,\cdots ,n\right\},\\
	&{ K}_2=\left\{i:\left|\frac{i}{n}-x\right|\ge \delta, i=0,1,2,\cdots ,n\right\}.\end{align*}
	Let $k\in { K}_1$. Since $\left| \xi(t,x)\right| < \varepsilon $,  from (\ref{xx}) we obtain
	\begin{align}
&	\left| \tilde{U}_n^{2}\left(\xi(t,x)(t-x)^6;x\right) \right|
		\leq\displaystyle\frac{\varepsilon n}{4}\left\{(n-1)\sum_{k=1}^{n-2}p_{n-2,k}(x)\int_0^1p_{n-2,k-1}(t)(t-x)^6 dt\right.\nonumber\\
	&+2(n-1)\sum_{k=1}^{n-1}p_{n-2,k-1}(x)\int_0^1p_{n-2,k-1}(t)(t-x)^6 dt\nonumber
	\end{align}\begin{align}
	&+(n-1)\sum_{k=2}^{n-1}p_{n-2,k-2}(x)\int_0^1p_{n-2,k-1}(t)(t-x)^6 dt
	+\left.(1-x)^{n-2}x^6+x^{n-2}(1-x)^6\right\}\nonumber
\\
	&\leq\displaystyle\frac{120\varepsilon x^3(1-x)^3(1+x-x^2)n^3}{(n+1)(n+2)(n+3)(n+4)(n+5)}+{ O}\left(n^{-3}\right).	\label{ec1}
	\end{align}
	Let $k\in { K}_2$. We denote $M=\displaystyle\sup_{0\leq t\leq 1} |\xi(t,x)|(t-x)^6$. Then
	 $$\left| \xi(t,x) \right|(t-x)^6\leq\displaystyle \frac{M}{\delta^6}\left(\frac{k}{n}-x \right)^6.$$
	 Moreover, the below upper bound is obtained
	\begin{align}\label{ec2}
&	\left| \tilde{ U}_n^{2}\left(\xi(t,x)(t-x)^6;x\right) \right|
		=\displaystyle\frac{Mn}{4\delta^6}\left\{\sum_{k=1}^{n-2}p_{n-2,k}(x)\left(\frac{k}{n}-x\right)^6
\right.\\
	&	+2\sum_{k=1}^{n-1}p_{n-2,k-1}(x)\left(\frac{k}{n}-x\right)^6+\sum_{k=2}^{n-1}p_{n-2,k-2}(x)\left(\frac{k}{n}-x\right)^6\nonumber\\
&	+\left.(1-x)^{n-2}x^6+x^{n-2}(1-x)^6\right\}
	=\displaystyle\frac{15Mx^3(1-x)^3}{n^2\delta^6}+{ O}\left(n^{-3}\right)={ O}\left(n^{-2} \right).\nonumber
	\end{align}
	From (\ref{ec1}) and (\ref{ec2}) the proof of theorem is completed.
\end{proof}

In the following we will improve the previous results considering a new\linebreak  genuine Bernstein-Durrmeyer operators that have order of approximation $\displaystyle{ O}\left(n^{-3}\right)$ defined as

\begin{align}\label{K3}
{ U}_{n}^{3}(f;x)&=(n-1)\displaystyle\sum_{k=1}^{n-1} p_{n,k}^{3}(x) \int_{0}^{1}p_{n-2,k-1}(t)f(t)dt\nonumber\\
&+\beta(x,n)(1-x)^{n-4}f(0)+\beta(1-x,n)x^{n-4}f(1),
\end{align}
where
\begin{align*} p_{n,k}^{3}(x)&=\beta(x,n)p_{n-4,k}(x)+\gamma(x,n)p_{n-4,k-1}(x)+\delta(x,n)p_{n-4,k-2}(x) \\
&+\gamma(1-x,n)p_{n-4,k-3}(x)+\beta(1-x,n)p_{n-4,k-4}(x)\end{align*}
and
\begin{align*}
& \beta(x,n)=\beta_4(n)x^4+\beta_3(n)x^3+\beta_2(n)x^2+\beta_1(n)x+\beta_0(n),\\
& \gamma(x,n)=\gamma_4(n)x^4+\gamma_3(n)x^3+\gamma_2(n)x^2+\gamma_1(n)x+\gamma_0(n),\\
&\delta(x,n)=\delta_0(n)(x(1-x))^2.
\end{align*}
Note that $\beta_i(n)$, $\gamma_i(n)$, $i=0,1,\dots,4$ and $\delta_0(n)$ are some unknown sequences. Denote $\tilde{ U}_n^{3}$  the operator (\ref{K3}) with 
\begin{align*}
& \beta_0(n)=1,\beta_1(n)=\displaystyle -4-\frac{4}{3}n, \beta_2(n)=\displaystyle 5+\frac{10}{3}n+\frac{1}{2}n^2, \beta_3(n)=-n^2-2n-2,\\
&\beta_4(n)=\displaystyle\frac{1}{2}n^2,\gamma_0(n)=0,
 \gamma_1(n)=\displaystyle 4+\frac{7}{3}n,\,\,\gamma_2(n)=\displaystyle -\frac{19}{3}n-2n^2-8,\\
 & \gamma_3(n)=4n^2+4n+4,
 \gamma_4(n)=-2n^2,\,\, \delta_0(n)=3n^2.
\end{align*}
Let $\tilde{\mu}_{n,k}^{3}(x):=\tilde{ U}_n^{3}\left((t-x)^k;x\right)$, $k=0,1,\dots$ the central moments of $\tilde{ U}_n^{3}$.
\begin{lemma}
	The modified genuine Bernstein-Durrmeyer operator $\tilde{ U}_n^{3}$ verify:
	\begin{itemize}
		\item [i)] $\tilde{\mu}_{n,1}^{3}(x)=\tilde{\mu}_{n,2}^{3}(x)=\tilde{\mu}_{n,3}^{3}(x)=0$;
		\item [ii)] $\tilde{\mu}_{n,4}^{3}(x)=\displaystyle \frac{4x(1-x)(39x^2-39x+10)}{\sum_{k=1}^3(n+k)}+{ O}\left(n^{-4}\right) $;
		\item [iii)] $\tilde{\mu}_{n,5}^{3}(x)=\displaystyle \frac{120(1-2x)x^2(1-x)^2n}{\sum_{k=1}^4(n+k)}+{ O}\left(n^{-4}\right)$;
		\item [iv)] $\tilde{\mu}_{n,6}^{3}(x)=\displaystyle\frac{120x^3(1-x)^3n^2}{\sum_{k=1}^{5}(n+k)}+{ O}\left(n^{-4}\right);  $
		\item [v)] $\tilde{\mu}_{n,7}^{3}(x)=\tilde{\mu}_{n,8}^{3}(x)=\displaystyle{ O} \left(n^{-4}\right); $
		\item[vi)] $\tilde{\mu}_{n,9}^{3}(x)=\tilde{\mu}_{n,10}^{3}(x)=\displaystyle{ O} \left(n^{-5}\right).$
	\end{itemize}
\end{lemma}
The asymptotic order of approximation of $\tilde{ U}_n^{3}$ to $f$ when $n$ goes to infinity is given in the following result:

\begin{theorem}
	If $f\in C^{10}[0,1]$ and $x\in [0,1]$, then 
	$ \tilde{ U}_n^{3}(f;x)-f(x)={ O}\left(n^{-3}\right). $
\end{theorem}

\section{Numerical Example}
Some numerical examples with illustrative graphics have been added to validate the theoretical results and also compare the rate of convergence by using Maple algorithms.

\begin{example} Let $g(x)=\displaystyle\sin(4\pi x)+4\sin\left(\frac{1}{4}\pi x\right)$, $n=10$, $\alpha_0(n)=\displaystyle\frac{n-1}{2n}$ and $\alpha_1(n)=\displaystyle\frac{1}{n}$.
	The convergence of the modified genuine Bernstein-Durrmeyer operators is illustrated in Figure \ref{fig:A1}.  Let $\varepsilon_n(f;x)=\left| g(x)-{ U}_n(g;x)\right|$ and $\varepsilon_n^{i}(g;x)=\left| g(x)-\tilde{ U}_n^{i}(g;x)\right|$, $i=1,2,3$ be the error of approximation for  genuine Bernstein-Durrmeyer operators and the modified genuine Bernstein-Durrmeyer operators, respectively. The error of approximation is illustrated in Figure \ref{fig:A2} and for this particular case  the approximation by the modified genuine Bernstein-Durrmeyer operators $\tilde{ U}_n^{i}$, $i=1,2,3$ is better then using classical genuine Bernstein-Durrmeyer operator ${ U}_n$.  Also, in Table \ref{table1} we computed the error of approximation for ${ U}_n$ and $\tilde{ U}_n^{i}$, $i=1,2,3$ at certain points.
	
\begin{minipage}{\linewidth}
	\centering
	\begin{minipage}{0.44\linewidth}
		\begin{figure}[H]
			\includegraphics[width=\linewidth]{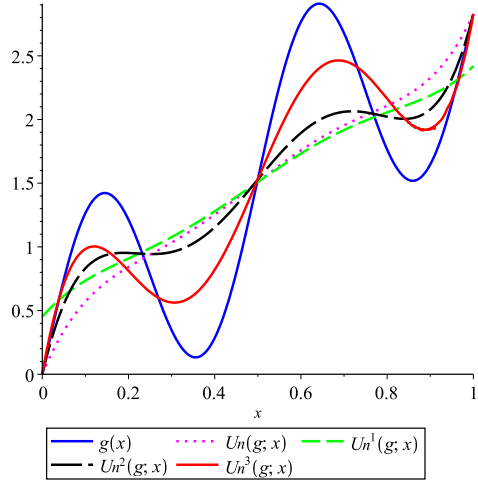}
			\caption{Convergence of the modified genuine Bernstein-Durrmeyer operators}\label{fig:A1}
		\end{figure}
	\end{minipage}
	\hspace{0.05\linewidth}
	\begin{minipage}{0.44\linewidth}
		\begin{figure}[H]
			\includegraphics[width=\linewidth]{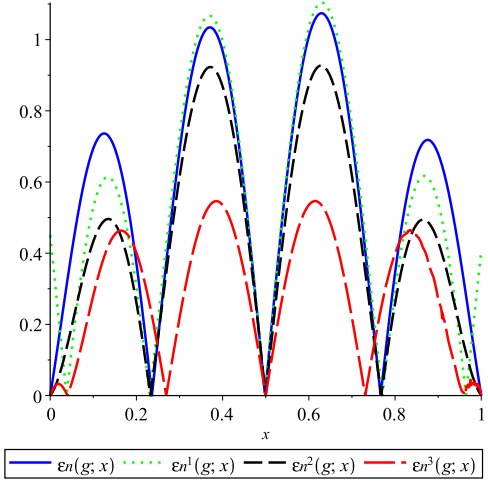}
			\caption{Error of approximation by the modified genuine Bernstein-Durrmeyer operators }\label{fig:A2}
		\end{figure}
	\end{minipage}
\end{minipage}
	\begin{tb}\label{table1}\centering
		{\it  Error of approximation  $\varepsilon_n$ and $\varepsilon_n^{i}$, $i=1,2,3$}
		
		$  $
		
		\begin{tabular}{|l|l|l|l|l|}\hline
			$x$&$\varepsilon_n(g;x)$&$\varepsilon_n^{1}(g;x)$&$\varepsilon_n^{2}(g;x)$&$\varepsilon_n^{3}(g;x)$ \\  \hline
			$0.10$ & $0.6945330253$  &  $0.5295762286$ & $0.4330397338$ & $0.2812687901$\\  \hline
			$0.15$ & $0.6942110247$ & $0.5933130652$ & $0.4833262822$ & $0.4529855958$ 	\\  \hline
			$0.25$ & $0.1580679909$ & $0.2097384650$ & $0.1632564535$ & $0.1344025490$ 	\\  \hline
			$0.30$ & $0.6879044013$ & $0.7312808277$ & $0.6136368403$ & $0.2177095258$ 	\\  \hline
			$0.40$ & $0.9681814067$ & $0.9948377556$ & $0.8683959826$ & $0.5352250570$ 	\\  \hline
			$0.50$ & $0.0213435310$ & $0.0213436470$ & $0.0022170860$ & $0.0000572989$ 	\\  \hline
			$0.55$ & $0.6243511367$ & $0.6381381849$ & $0.5355754034$ & $0.3459331618$ 	\\  \hline
			$0.70$ & $0.7234749900$ & $0.7601315856$ & $0.6172299990$ & $0.2177088936$ 	\\  \hline
			$0.75$ & $0.1896839607$ & $0.2308635388$ & $0.1663414950$ & $0.1338132004$ 	\\  \hline
			$0.90$ & $0.6796558035$ & $0.5415051330$ & $0.4322424640$ & $0.2822400676$ 	\\  \hline
			$0.95$ & $0.3964880751$ & $0.1546334112$ & $0.194510818$ & $0.0192152324$ 	\\  \hline
		\end{tabular}
	\end{tb}
\end{example}

\begin{example}
	Let 
	$ g(x)=\left| x-\frac{1}{4} \right|\cos(4\pi x). $  For $n=10$, $\alpha_0(n)=\displaystyle\frac{n-1}{2n}$ and $\alpha_1(n)=\displaystyle\frac{1}{n}$, the convergence of the modified genuine Bernstein-Durrmeyer operators to  $g(x)$ is shown in Figure \ref{A3}. The errors of approximation  $\varepsilon_n$ and $\varepsilon_n^{i}$, $i=1,2,3$ are illustrated in Figure \ref{A4}. 
Also, in Table \ref{table2} we computed the error of approximation for ${ U}_n$ and $\tilde{ U}_n^{i}$, $i=1,2,3$ at certain points.

\begin{minipage}{\linewidth}
	\centering
	\begin{minipage}{0.44\linewidth}
		\begin{figure}[H]
			\includegraphics[width=\linewidth]{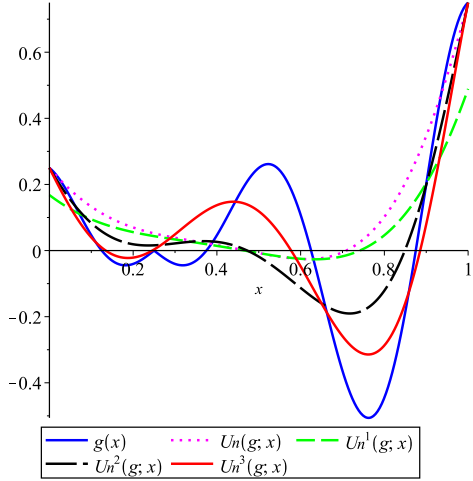}
			\caption{Convergence of the modified genuine Bernstein-Durrmeyer operators}\label{A3}
		\end{figure}
	\end{minipage}
	\hspace{0.05\linewidth}
	\begin{minipage}{0.43\linewidth}
		\begin{figure}[H]
			\includegraphics[width=\linewidth]{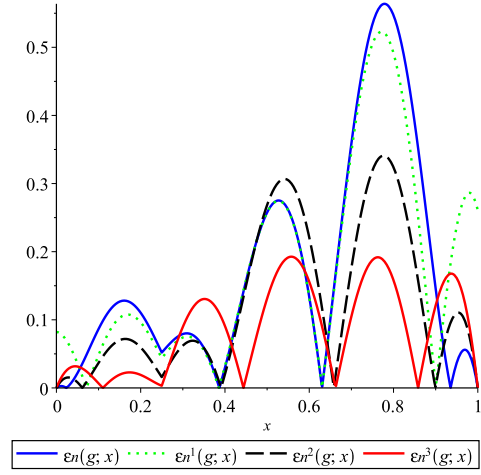}
			\caption{Error of approximation  $\varepsilon_n$ and $\varepsilon_n^{i}$, $i=1,2,3$ }\label{A4}
		\end{figure}
	\end{minipage}
\end{minipage}
\begin{tb}\label{table2}\centering
	{\it  Error of approximation  $\varepsilon_n$ and $\varepsilon_n^{i}$, $i=1,2,3$}
	
	$  $
	
	\begin{tabular}{|l|l|l|l|l|}\hline
		$x$&$\varepsilon_n(g;x)$&$\varepsilon_n^{1}(g;x)$&$\varepsilon_n^{2}(g;x)$&$\varepsilon_n^{3}(g;x)$ \\  \hline
		$0.10$ & $0.0864633606$  &  $0.0482073753$ & $0.0367065261$ & $0.0067324201$\\  \hline
		$0.20$ & $0.1103333385$  &  $0.0961865298$ & $0.0594456124$ & $0.0191120699$\\  \hline
		$0.25$ & $0.0518922812$  &  $0.0436031934$ & $0.0158446516$ & $0.0031661470$\\  \hline
		$0.50$ & $0.2575784301$  &  $0.2575784294$ & $0.2685059808$ & $0.1312442040$\\  \hline
		$0.55$ & $0.2614871004$  &  $0.2604287058$ & $0.3049134784$ & $0.1915072058$\\  \hline
		$0.65$ & $0.1006465455$  &  $0.0976735308$ & $0.0362205580$ & $0.0365602345$\\  \hline
		$0.70$ & $0.3587223234$  &  $0.3466837357$ & $0.1760351064$ & $0.1049635359$\\  \hline
		$0.85$ & $0.3901466578$  &  $0.2963782143$ & $0.1888707949$ & $0.0267144753$\\  \hline
		$0.90$ & $0.1458116388$  &  $0.0029193836$ & $0.0022766453$ & $0.1211619468$\\  \hline
	\end{tabular}
\end{tb}
\end{example}

\begin{example}
	Let  $g(x)=\left(x-\frac{1}{4}\right)\sin(2\pi x)$. The behaviours of the approximations ${ U}_n(g;x)$, $\tilde{ U}_n^{i}(g;x)$, $i=1,2,3$ and their errors  $\varepsilon_n(g;x)$, $\varepsilon_n^{i}(g;x)$ for $n=5,7,10$, $\alpha_0(n)=\displaystyle\frac{n-1}{2n}$, $\alpha_1(n)=\displaystyle\frac{1}{n}$ are illustrated in the Figures \ref{A5}-\ref{A12}. 
Also, in Tables \ref{table3}-\ref{table6} we computed the error of approximation for ${ U}_n$ and $\tilde{ U}_n^{i}$, $i=1,2,3$, for $n=5,7,10$ at certain points.
\end{example}

\begin{minipage}{\linewidth}
\centering
\begin{minipage}{0.44\linewidth}
	\begin{figure}[H]
		\includegraphics[width=\linewidth]{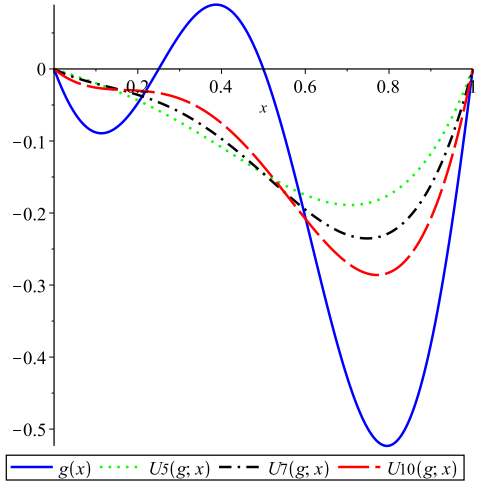}
		\caption{Convergence of the genuine Bernstein-Durrmeyer operators  ${ U}_n$  }\label{A5}
	\end{figure}
\end{minipage}
\hspace{0.05\linewidth}
\begin{minipage}{0.43\linewidth}
	\begin{figure}[H]
		\includegraphics[width=\linewidth]{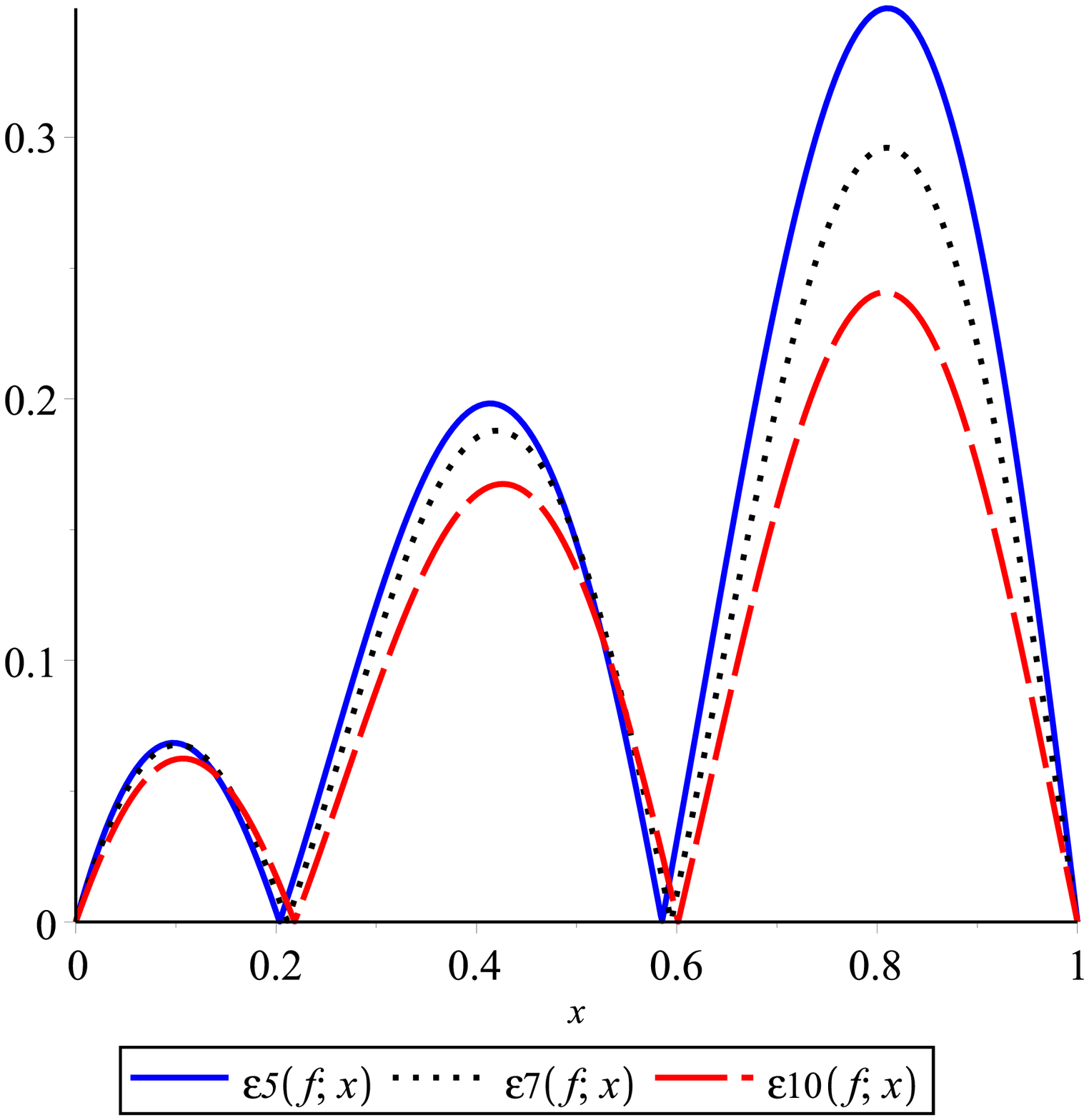}
		\caption{Error of approximation  $\varepsilon_n$  }\label{A6}
	\end{figure}
\end{minipage}
\end{minipage}

$$  $$

\begin{tb}\label{table3}\centering
{\it  Error of approximation  $\varepsilon_n$ for $n=5,7,10$}

$   $

\begin{tabular}{|l|l|l|l|}\hline
	$x$&$\varepsilon_5(g;x)$&$\varepsilon_{7}(g;x)$&$\varepsilon_{10}(g;x)$ \\  \hline
	$0.10$ & $0.0684403678$  &  $0.0678031572$ & $0.0622151561$ \\  \hline
	$0.15$ & $0.0499431989$  &  $0.0531270433$ & $0.0521857868$ \\  \hline
	$0.25$ & $0.0578940551$  &  $0.0463235563$ & $0.0334294267$ \\  \hline
	$0.30$ & $0.1212531914$  &  $0.1072975170$ & $0.0886731506$ \\  \hline
	$0.40$ & $0.1969741168$  &  $0.1851488999$ & $0.1634615511$ \\  \hline
	$0.65$ & $0.1391012080$  &  $0.1078116400$ & $0.0816688494$ \\  \hline
	$0.70$ & $0.2391101684$  &  $0.1978330494$ & $0.1589711006$ \\  \hline
	$0.85$ & $0.3328374955$  &  $0.2808786960$ & $0.2265106612$ \\  \hline
	$0.90$ & $0.2644985309$  &  $0.2206728121$ & $0.1747401754$ \\  \hline
\end{tabular}
\end{tb}

\newpage


\begin{minipage}{\linewidth}
	\centering
	\begin{minipage}{0.43\linewidth}
		\begin{figure}[H]
			\includegraphics[width=\linewidth]{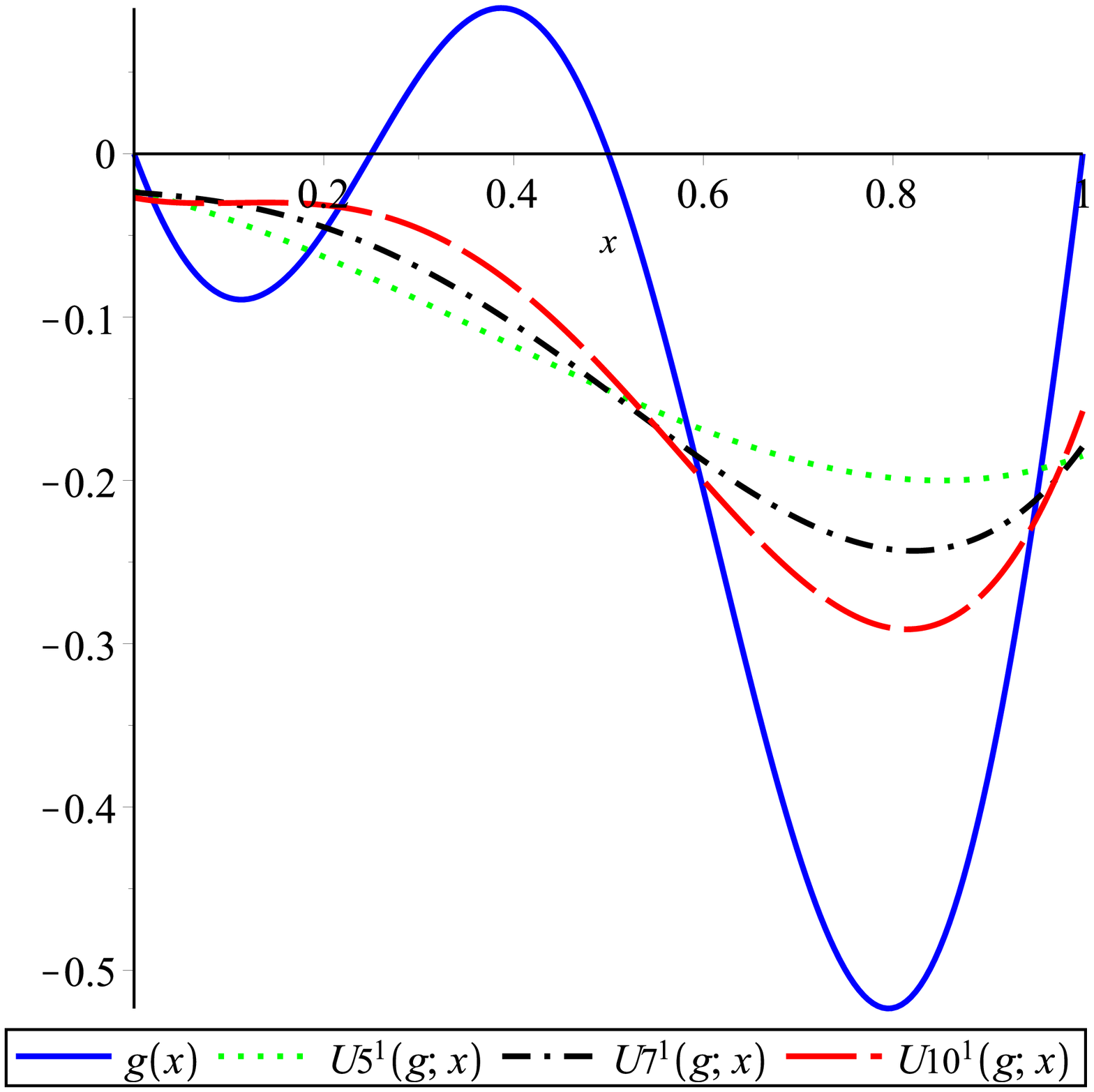}
			\caption{Convergence of the modified genuine Bernstein-Durrmeyer operators  ${ U}_n^{1}$ }\label{A7}
		\end{figure}
	\end{minipage}
	\hspace{0.05\linewidth}
	\begin{minipage}{0.43\linewidth}
		\begin{figure}[H]
			\includegraphics[width=\linewidth]{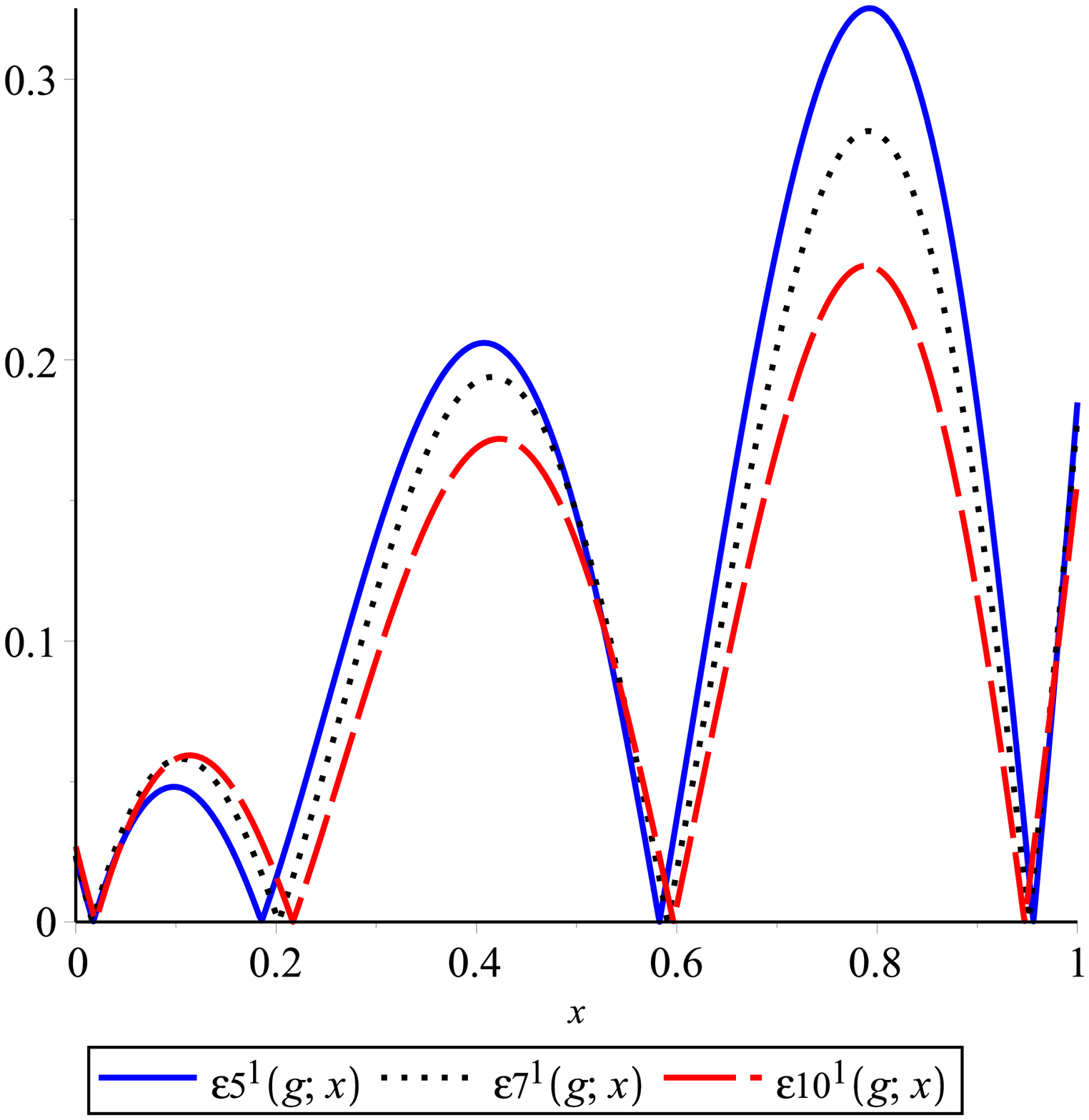}
			\caption{Error of approximation  $\varepsilon_n^{1}$  }\label{A8}
		\end{figure}
	\end{minipage}
\end{minipage}

$$  $$

\begin{tb}\label{table4}\centering
	{\it  Error of approximation  $\varepsilon_n^{1}$ for $n=5,7,10$}
	
	$  $
	
	\begin{tabular}{|l|l|l|l|}\hline
		$x$&$\varepsilon_5(g;x)$&$\varepsilon_{7}(g;x)$&$\varepsilon_{10}(g;x)$ \\  \hline
		$0.25$ & $0.0759640671$  &  $0.0559388716$ & $0.0362644403$ \\  \hline
		$0.30$ & $0.1371072386$  &  $0.1171805319$ & $0.0932932589$ \\  \hline
		$0.35$ & $0.1844145994$  &  $0.1666950554$ & $0.1413524766$ \\  \hline
		$0.40$ & $0.2057325249$  &  $0.1923019252$ & $0.1687116616$ \\  \hline
		$0.50$ & $0.1448255593$  &  $0.1453442003$ & $0.1349064978$ \\  \hline
		$0.60$ & $0.0366918591$  &  $0.0179868291$ & $0.0057986703$ \\  \hline
		$0.70$ & $0.2402023843$  &  $0.2045318798$ & $0.1686629702$ \\  \hline
		$0.80$ & $0.3245837174$  &  $0.2807124726$ & $0.2325391597$ \\  \hline
		$0.90$ & $0.1835617165$  &  $0.1497215453$ & $0.1155690815$ \\  \hline
	\end{tabular}
\end{tb}

\newpage


\begin{minipage}{\linewidth}
	\centering
	\begin{minipage}{0.44\linewidth}
		\begin{figure}[H]
			\includegraphics[width=\linewidth]{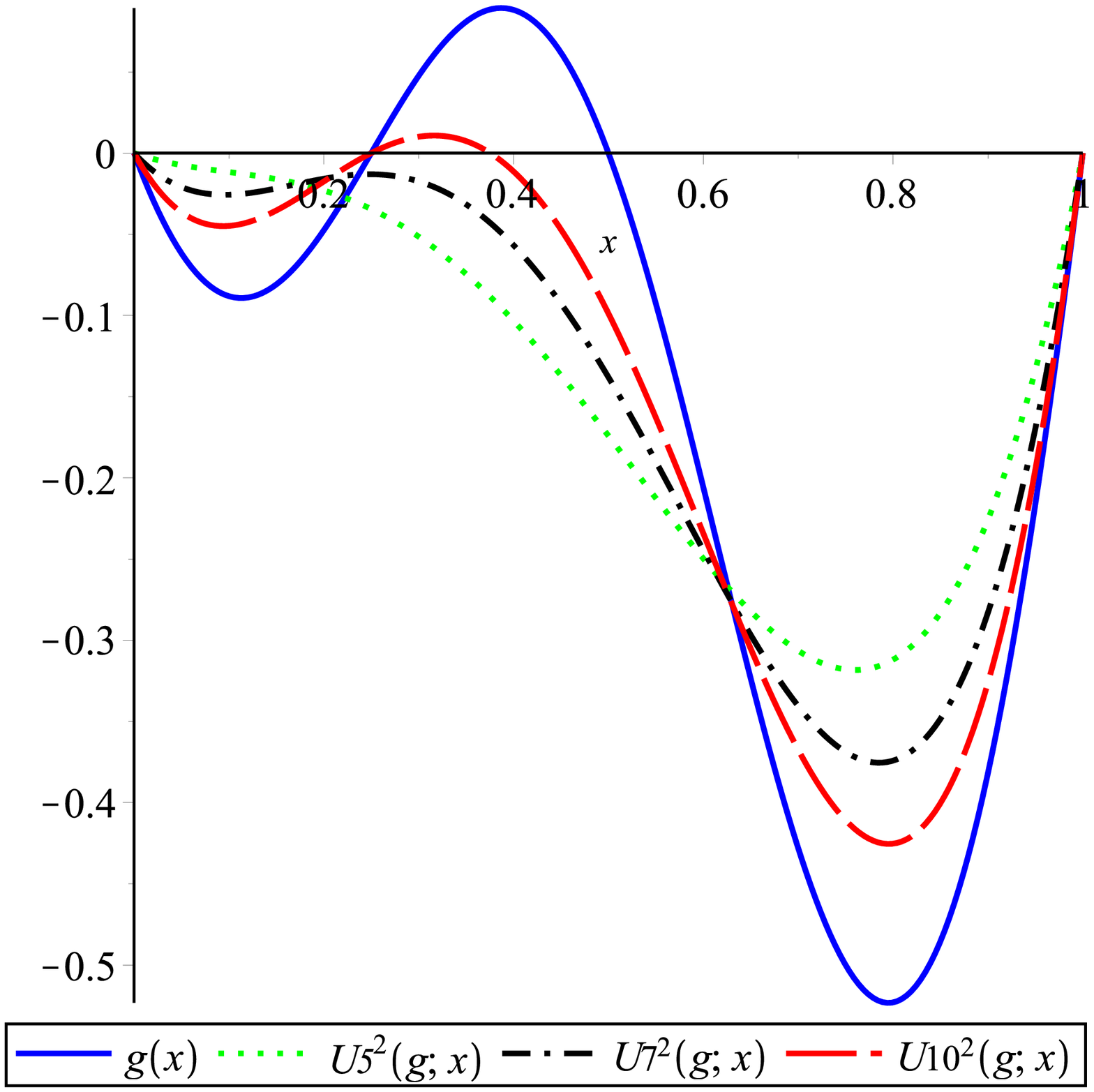}
			\caption{Convergence of the modified genuine Bernstein-Durrmeyer operators  ${ U}_n^{2}$ }\label{A9}
		\end{figure}
	\end{minipage}
	\hspace{0.05\linewidth}
	\begin{minipage}{0.43\linewidth}
		\begin{figure}[H]
			\includegraphics[width=\linewidth]{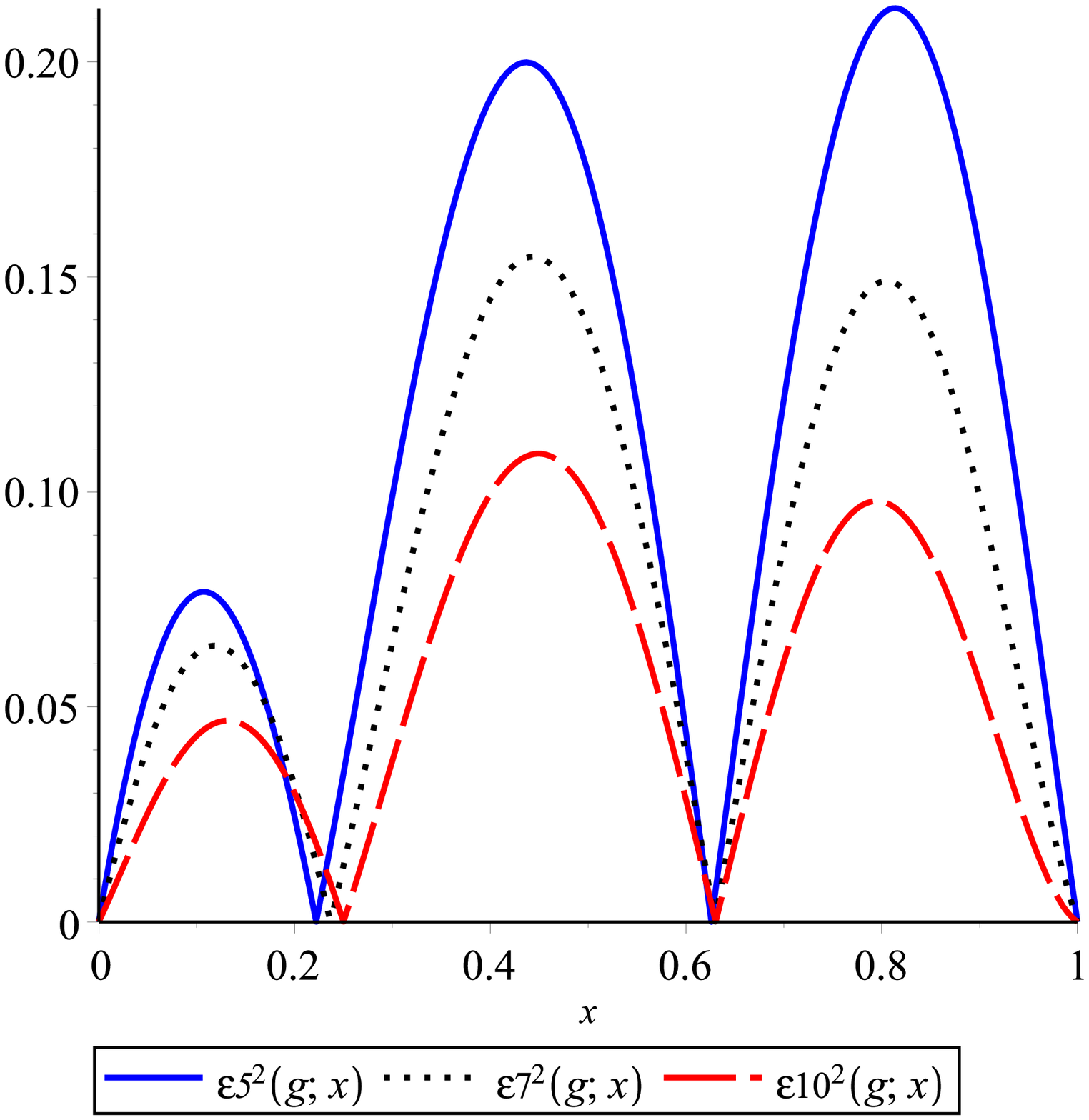}
			\caption{Error of approximation  $\varepsilon_n^{2} $ }\label{A10}
		\end{figure}
	\end{minipage}
\end{minipage}

$$   $$

\begin{tb}\label{table5}\centering
	{\it  Error of approximation  $\varepsilon_n^{2}$ for $n=5,7,10$}
	
	$  $
	
	\begin{tabular}{|l|l|l|l|}\hline
		$x$&$\varepsilon_5(g;x)$&$\varepsilon_{7}(g;x)$&$\varepsilon_{10}(g;x)$ \\  \hline
		$0.10$ & $0.0764438176$  &  $0.0624790953$ & $0.0432951373$ \\  \hline
		$0.25$ & $0.0345024248$  &  $0.0130565516$ & $0.0000245123$ \\  \hline
		$0.30$ & $0.0990500279$  &  $0.0650077665$ & $0.0375060681$ \\  \hline
		$0.40$ & $0.1914411452$  &  $0.1449016157$ & $0.0994953904$ \\  \hline
		$0.50$ & $0.1740970203$  &  $0.1380269700$ & $0.0987547084$ \\  \hline
		$0.60$ & $0.0439765069$  &  $0.0388439363$ & $0.0284303826$ \\  \hline
		$0.70$ & $0.1215101036$  &  $0.0874936435$ & $0.0605939578$ \\  \hline
		$0.80$ & $0.2110574499$  &  $0.1488955530$ & $0.0977772009$ \\  \hline
		$0.90$ & $0.1561838905$  &  $0.0989074740$ & $0.0557077924$ \\  \hline
	\end{tabular}
\end{tb}

\newpage

\begin{minipage}{\linewidth}
	\centering
	\begin{minipage}{0.45\linewidth}
		\begin{figure}[H]
			\includegraphics[width=\linewidth]{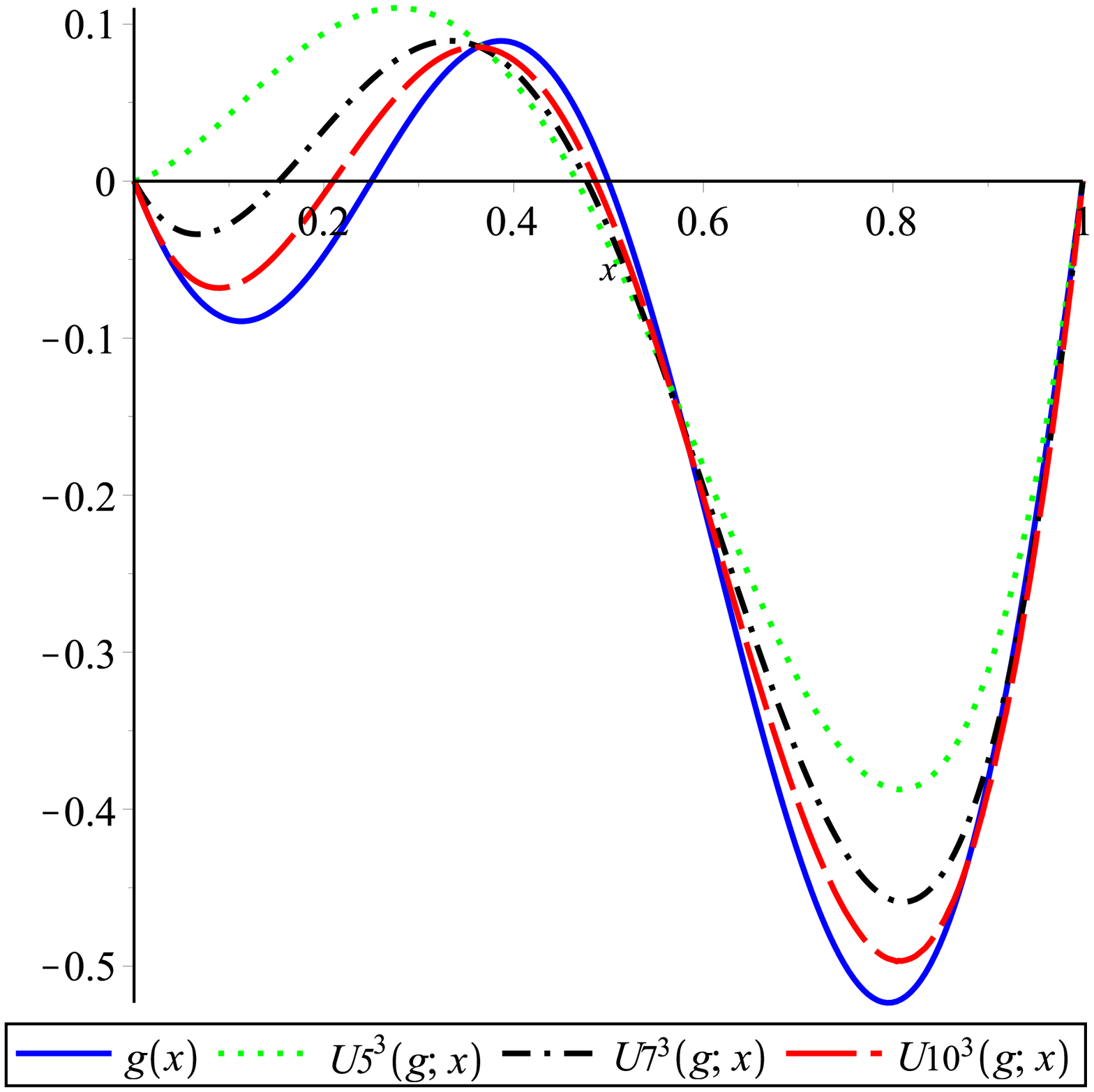}
			\caption{Convergence of the modified genuine Bernstein-Durrmeyer operators  ${ U}_n^{3}$ }\label{A11}
		\end{figure}
	\end{minipage}
	\hspace{0.05\linewidth}
	\begin{minipage}{0.42\linewidth}
		\begin{figure}[H]
			\includegraphics[width=\linewidth]{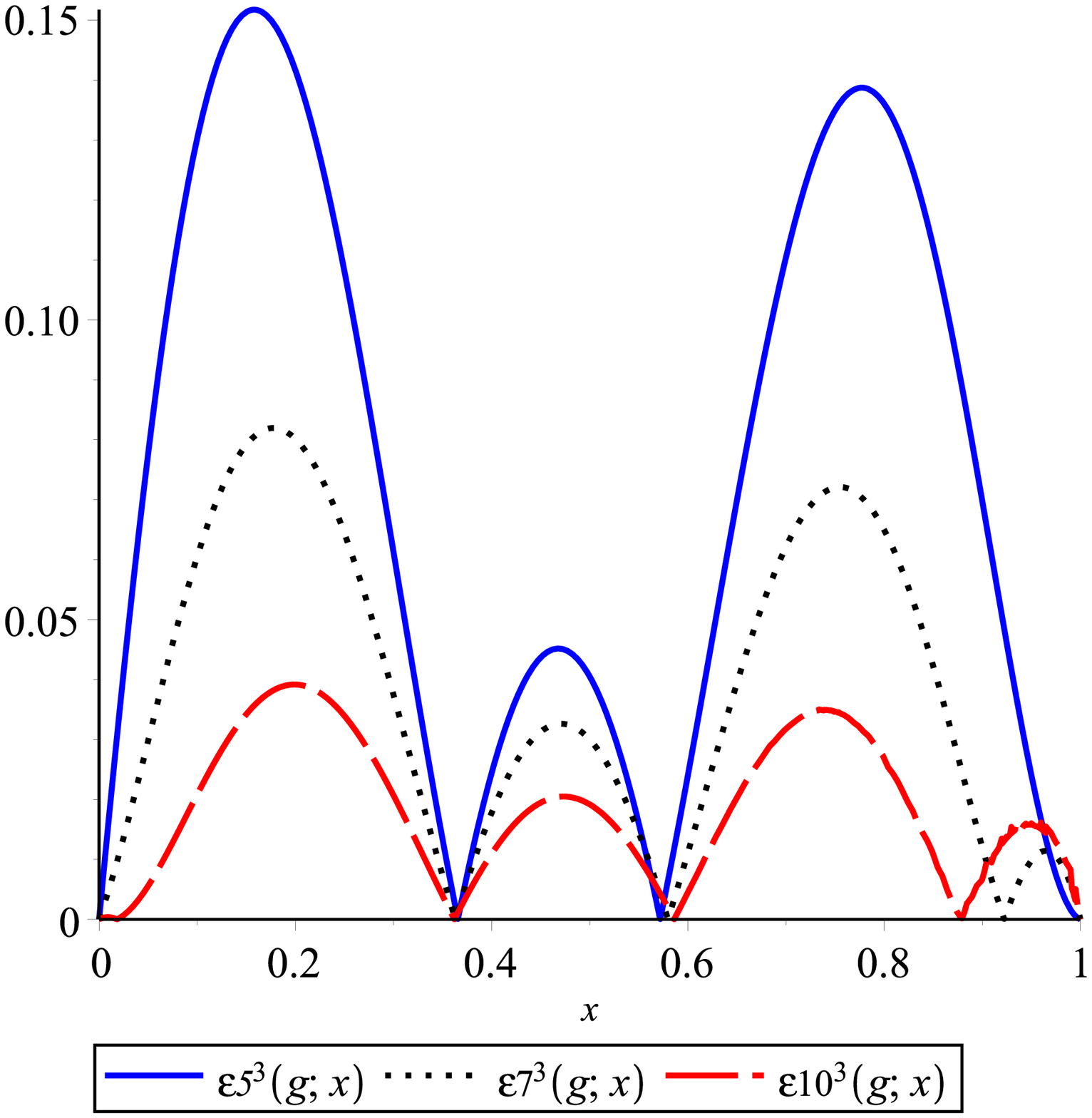}
			\caption{Error of approximation  $\varepsilon_n^{3}$  }\label{A12}
		\end{figure}
	\end{minipage}
\end{minipage}
$  $

\begin{tb}\label{table6}\centering
	{\it  Error of approximation  $\varepsilon_n^{3}$ for $n=5,7,10$}
	
	$  $
	
	\begin{tabular}{|l|l|l|l|}\hline
		$x$&$\varepsilon_5(g;x)$&$\varepsilon_{7}(g;x)$&$\varepsilon_{10}(g;x)$ \\  \hline
		$0.10$ & $0.130304351$  &  $0.060532301$ & $0.020830446$ \\  \hline
		$0.20$ & $0.141691111$  &  $0.080159674$ & $0.039173543$ \\  \hline
		$0.30$ & $0.061291162$  &  $0.037462997$ & $0.020686153$ \\  \hline
		$0.40$ & $0.024416705$  &  $0.017774084$ & $0.010936011$ \\  \hline
		$0.45$ & $0.043687998$  &  $0.031315020$ & $0.019471529$ \\  \hline
		$0.50$ & $0.040509305$  &  $0.029993035$ & $0.019169082$ \\  \hline
		$0.60$ & $0.024027791$  &  $0.011127560$ & $0.004477202$ \\  \hline
		$0.70$ & $0.110487435$  &  $0.062167445$ & $0.031839458$ \\  \hline
		$0.80$ & $0.136082209$  &  $0.064976021$ & $0.026722598$ \\  \hline
		$0.90$ & $0.069748506$  &  $0.012020135$ & $0.006394035$ \\  \hline
	\end{tabular}
\end{tb}

$$  $$

{\bf Acknowledgments.} The work  was financed from Lucian Blaga University of Sibiu research grants LBUS-IRG-2018-04.

$   $

\noindent{\bf References}

\end{document}